\documentclass[a4paper,12pt]{amsart}

  \usepackage{amssymb,amsthm}
\usepackage{graphicx,color}    
\usepackage{changes}
\usepackage{url} 

\newtheorem{theorem}{Theorem}[section]

\newtheorem{proposition}[theorem]{Proposition}

\theoremstyle{definition}
\newtheorem{definition}[theorem]{Definition}
\newtheorem{example}[theorem]{Example}
\newtheorem{remark}[theorem]{Remark}

\def\r{\mathbb R}
\def\h{\mathbb H}
\def\l{\mathbb L}
\def\s{\mathbb S}

\begin{document}

\title[Ruled stationary surfaces in Lorentz-Minkowski space]{Classification of ruled surfaces in Lorentz-Minkowski space that are stationary for the moment of inertia}
\author{Muhittin Evren Aydin}
\address{Department of Mathematics, Faculty of Science, Firat University, Elazig,  23200 Turkey}
\email{meaydin@firat.edu.tr}
\author{Rafael L\'opez}
\address{Departamento de Geometr\'{\i}a y Topolog\'{\i}a Universidad de Granada 18071 Granada, Spain}
\email{rcamino@ugr.es}
   \begin{abstract}
In this paper, we study hypersurfaces  in Lorentz-Minkowski space $\l^{n+1}$ that are stationary for the  moment of inertia with respect to the origin. After giving examples and applications of the maximum principle, we classify, in $\l^3$, all stationary surfaces that are ruled surfaces. We prove that planes are the only cylindrical stationary surfaces. If the surface is not cylindrical, the classification depends on the causal character of the rulings. In addition, we provide explicit parametrizations of all ruled stationary surfaces.  
\end{abstract}

\maketitle
\section{Introduction of the problem}

In 1744, Euler proposed the study of curves $\gamma$ in the Euclidean plane that are critical points of the energy functional
\begin{equation}\label{Eu1}
E_{\alpha}[\gamma]=\int_\gamma |\gamma(s)|^\alpha\, ds,
\end{equation}
where $\alpha\in\r$ is a parameter,  and $s$ is the arc-length parameter \cite{eu}.  When $\alpha=0$, the energy $E_0[\gamma]$ is the  length of $\gamma$. For  $\alpha=2$, the energy $E_2[\gamma]$ corresponds to the moment of inertia of $\gamma$ with respect to the origin   of $\r^2$. Critical points of $E_\alpha$ are called {\it $\alpha$-stationary curves}, and they are characterized in terms of the curvature $\kappa$ of the curve, namely,  
\begin{equation}\label{Eu2}
\kappa =\alpha\frac{\langle N ,\gamma \rangle}{|\gamma|^2},
\end{equation}
where  $N$ is the unit normal vector of $\gamma$.  Notice that the differentiability of the functional $E_\alpha$ requires  that $\gamma$ does not cross the origin of $\r^2$. 

Recently, Dierkes and Huisken extended this problem to arbitrary dimensions, considering the energy functional
\begin{equation}\label{Eu3}
E_\alpha[\Sigma]=\int_\Sigma |p|^\alpha\, d\Sigma,
\end{equation}
where $\Sigma$ is a hypersurface of $\r^{n+1}$ \cite{dh} and $p\in \Sigma$. The corresponding Euler-Lagrange equation is similar to \eqref{Eu2}, where the curvature $\kappa$ and the unit normal $N$ replaced by the mean curvature $H$ and the unit normal vector of $\Sigma$, respectively. Recent progress in the study of $\alpha$-stationary hypersurfaces can be found in   \cite{di,dl1,dl2,lo2}.

In this paper, we extend this problem in Lorentz-Minkowski space $\l^{n+1}$.  The space $\l^{n+1}$ is $(\r^{n+1},\langle,\rangle)$, where  the metric is $\langle,\rangle=dx_1^2+\ldots+dx_n^2-dx_{n+1}^2$ and $(x_1,\ldots,x_{n+1})$ are the canonical coordinates of $\r^{n+1}$. In comparison with the Euclidean case, some important differences appear in the Lorentzian context. First, in order for $d\Sigma$ in \eqref{Eu3} to make sense, it is necessary to restrict the study to spacelike hypersurfaces. 

The second difference concerns the term $|p|^\alpha$ of the energy $E_\alpha$ in \eqref{Eu3}. In $\l^{n+1}$, the term $\langle p,p\rangle$ is not necessarily positive. Thus by $|p|$ we will indicate $|p|=\sqrt{|\langle p,p\rangle|}$ to ensure that the  power $|p|^\alpha$ is well-defined. As in the Euclidean case, we require that $0\not\in \Sigma$ to guarantee differentiability of the functional $E_\alpha$. However, in the Lorentzian setting we must also impose that $\Sigma$ does not intersect the lightlike cone $\mathcal{C}=\{p\in\l^{n+1}\colon \langle p,p\rangle=0\}$ because of the denominator in \eqref{Eu2}. 

The lightlike cone $\mathcal{C}$ separates $\l^{n+1}$ in two sets, 
$$\mathcal{C}^+=\{p\in\l^{n+1}\colon \langle p,p\rangle <0\},\quad \mathcal{C}^-=\{p\in\l^{n+1}\colon \langle p,p\rangle >0\}.$$
If $n\geq 2$, the set $\mathcal{C}^-$ is connected, but $\mathcal{C}^+$ has two components.   Consequently,  it is necessary  to distinguish if $\Sigma$ is included in $\mathcal{C}^+$ or in $\mathcal{C}^-$.

Following standard calculus of variations (see Sect. \ref{s2}), a spacelike hypersurface $\Sigma$ is a critical point of the energy $E_\alpha$ if and only if the mean curvature $H$ of $\Sigma$ satisfies
\begin{equation}\label{ss}
H(p)= -\alpha \frac{ \langle N(p),p\rangle}{|\langle p,p\rangle|}, \quad p \in \Sigma.
\end{equation}
To distinguish whether $\Sigma$ lies in $\mathcal{C}^+$  or in $ \mathcal{C}^-$, we can write \eqref{ss} as 
\begin{equation}\label{eq00}
H(p)= \epsilon\alpha \frac{ \langle N(p),p\rangle}{\langle p,p\rangle}, \quad   p \in \Sigma,
\end{equation}
where $\epsilon = 1$ if $\Sigma \subset \mathcal{C}^+$ and $\epsilon = -1$ if $\Sigma \subset \mathcal{C}^-$.

Eq. \eqref{ss} gives an     interesting observation. Although the problem was initially posed only for spacelike hypersurfaces, in order for the integral in \eqref{Eu3} to make sense, all terms of Eq.  \eqref{ss} do  for timelike hypersurfaces. This allows to extend the notion of $\alpha$-stationary hypersurface to the timelike case.

\begin{definition} \label{dss}
Let $\Sigma$ be a spacelike or timelike hypersurface in $\l^{n+1}$ satisfying $\Sigma \cap \mathcal{C} = \emptyset$. We call $\Sigma$ an $\alpha$-stationary hypersurface if  
\begin{equation}
H(p)=-\alpha \frac{ \langle N(p),p\rangle}{|\langle p,p\rangle|}, \quad p \in \Sigma. \label{eq1}
\end{equation}
\end{definition}
 
If the value of $\alpha$ is not to be emphasized, we simply say stationary hypersurfaces.

Recently, the authors of the present paper have studied stationary hypersurfaces in the particular case $n=1$, that is, stationary curves in the Lorent-Minkowski plane  obtaining explicit parametrizations of these curves \cite{al}.

 The   objective of this paper is to classify all ruled stationary surfaces in $\l^3$. First, in Sect. \ref{s2} we derive the Euler-Lagrange equation of the energy $E_\alpha$ defined in \eqref{Eu3}, and we extend to timelike hypersurfaces. In this section, we also show examples in the class of hyperbolic hyperplanes and pseudospheres.  In Sect. \ref{s22} we give an application of the maximum principle for stationary spacelike hypersurfaces and, via inversions, relate the value $\alpha$ for different stationary hypersurfaces. The classification of ruled stationary surfaces in $\l^3$ is done in Sects. \ref{s3}, \ref{s4} and \ref{s5}. The discussion depends on whether the surface is cylindrical and on whether the rulings are lightlike or not.  Suppose that a ruled surface $\Sigma$ is parametrized by $X(s,t)=\gamma(s)+tw(s)$, $s\in I$, $t\in\r$. Notice that for a spacelike ruled surface, it is necessary (but not sufficient) that both the base curve $\gamma$  and the rulings $w$ must be spacelike. However, in case that the surface is timelike, the situation is richer. As summary of these results, we prove:
 \begin{enumerate}
 \item Vector planes are the only cylindrical stationary surfaces in $\l^3$ (Thm. \ref{cy} in Sect. \ref{s3}).
 \item If $w$ and $w'$ are not lightlike, then $\alpha\in\{-1,1\}$. Moreover, the curve $w=w(s)$ is a geodesic of   a pseudosphere (Thm. \ref{th-nli} in Sect. \ref{s4}). 
 \item Vector planes are the only ruled stationary surfaces in $\l^3$ such that  $w$ is not lightlike and $w'$ is lightlike (Thm. \ref{th-nli2} in Sect. \ref{s4}). 
 \item If $w$ is lightlike, then $\alpha\in\{-4, 2,4\}$. Moreover, if $\alpha= 2$, then $\gamma(s)=w'(s)$  (Thm. \ref{th-li} in Sect. \ref{s5}). 
 \end{enumerate}

\section{The Euler-Lagrange equation and examples}\label{s2}

 In this section, we begin with the derivation of the Euler-Lagrange equation of the energy \eqref{Eu3}, next we show examples of stationary hypersurfaces and finally, we give an application of the maximum principle. 

First, we obtain the Euler-Lagrange equation for \eqref{Eu3} using standard arguments from the calculus of variations and we point out the differences with the Euclidean case.  Let $\Sigma$ be a spacelike hypersurface in $\l^{n+1}$. Since the arguments are local, we may assume that $\Sigma$ is locally the graph of a function $u \colon \Omega \subset \mathbb{R}^n \to \mathbb{R}$, where we identify $\r^n$ with the hyperplane $x_{n+1}=0$. Set $p=(x_1,\ldots,x_{n+1})=(q,x_{n+1})$.  Consider the immersion of $\Sigma$ given by 
$$X\colon\Omega\to\l^{n+1},\quad X(q)=(q,u(q)).$$

As a first step, we suppose  $\Sigma\subset\mathcal{C}^+$. This implies $\langle p,p\rangle=|q|^2-u(q)^2<0$. Hence, $|p|=\sqrt{u^2-|q|^2}$. Consider the  unit normal vector on $\Sigma$ given by 
\begin{equation}\label{n}
N=\frac{(Du,1)}{\sqrt{1-|Du|^2}}.
\end{equation} 
Since $\Sigma$ is spacelike, then $|Du|^2<1$. The energy functional  \eqref{Eu3} can be written as
$$E_\alpha[\Sigma]=\int_\Omega (u^2-|q|^2)^{\alpha/2}\sqrt{1-|Du|^2}:=\int_\Omega J(q,u,Du).$$
 As usual, the Euler-Lagrange equation is given by the expression
$$\frac{\partial J}{\partial u}=\sum_{i=1}^n\left(\frac{\partial J}{\partial u_i}\right)_{x_i},$$
where $u_i:=\frac{\partial u}{\partial x_i}$. A straightforward computation gives
\begin{equation}\label{eq0}
\alpha \frac{\sum_{i=1}^n x_i u_i-u }{\sqrt{1-|Du|^2}}=-(u^2-|q|^2)\sum_{i=1}^n\left(\frac{u_i}{\sqrt{1-|Du|^2}}\right)_{x_i}.
\end{equation}
The first parenthesis on the right-hand side is $-\langle X,X\rangle$ and the second one is  the mean curvature $H$ of $\Sigma$ with respect to the orientation \eqref{n}.  Along the paper, the mean curvature $H$ is understood as the trace of the second fundamental form of $\Sigma$. On the other hand, we have
$$\langle N,X\rangle=\frac{\sum_{i=1}^n x_i u_i-u }{\sqrt{1-|Du|^2}}.$$
Definitively, Equation \eqref{eq00} can be written as
\begin{equation}\label{hh}
H=\alpha\frac{\langle N,X\rangle}{\langle X,X\rangle},
\end{equation}
which coincides with \eqref{eq1} because $\langle X,X\rangle<0$. 

If $\Sigma$ lies in $\mathcal{C}^-$, the   energy becomes
$$E_\alpha[\Sigma]=\int_\Omega (|q|^2-u^2)^{\alpha/2}\sqrt{1-|Du|^2},$$
and we derive the same expression \eqref{hh} except with a minus sign on the right-hand side. Thus, the Euler-Lagrange equation  coincides again with   \eqref{eq1}. 

If $\alpha=0$, then $E_0$ is the $n$-area of a spacelike hypersurface. It is known that the critical points of the $n$-area are maximal hypersurfaces, that is, hypersurfaces whose mean curvature $H$ vanishes. This case will not be considered further.

Although the energy \eqref{Eu3} does not make sense for timelike hypersurfaces, Eq. \eqref{eq1} still holds. Recall that the mean curvature $H$ is   the trace of the second fundamental form of $\Sigma$ and that  the Weingarten map may fail to be diagonalizable. So principal curvatures are not, in general, well defined. In the case that it is diagonalizable, the eigenvalues are called principal curvature and $H$ coincides with their sum.

We can do  a similar argument for timelike hypersurfaces and we will see that we obtain the same Eq.  \ref{eq1}. Let  $\Sigma$ be a timelike hypersurface given locally as $u=u(q)$.  The timelike condition writes as $1<|Du|^2$ and the mean curvature $H$ and the unit normal $N$ are   
$$H=\mbox{div}\left(\frac{Du}{\sqrt{|Du|^2-1}}\right)=\sum_{i=1}^n\left(\frac{u_i}{\sqrt{|Du|^2-1}}\right)_{x_i},\quad N=\frac{(-Du,-1)}{\sqrt{|Du|^2-1}}.$$
Suppose that   $\Sigma$ is contained in   $\mathcal{C}^-$. Then $|q|^2-u^2>0$ and define  the energy 
$$E_\alpha[\Sigma]=\int_\Omega (|q|^2-u^2)^{\alpha/2}\sqrt{|Du|^2-1}.$$
We deduce that  $\Sigma$ is a critical point of $E_\alpha$ if and only if 
$$\mbox{div}\left(\frac{Du}{\sqrt{|Du|^2-1}}\right)=\alpha\frac{\langle Du,q\rangle-u}{(|q|^2-u^2)\sqrt{|Du|^2-1}}.$$
Since $\langle N,X\rangle=-(\langle Du,q\rangle-u)/\sqrt{|Du|^2-1}$, we obtain 
$$H=-\alpha\frac{\langle N,X\rangle}{|q|^2-u^2}=-\alpha\frac{\langle N,X\rangle}{\langle X,X\rangle}.$$
Thus this equation coincides with \eqref{eq1}. In case that $\Sigma\subset\mathcal{C}^+$, the argument is similar, obtaining \eqref{eq1} again.

\begin{remark} We point out that   linear isometries and dilations of $\l^{n+1}$      preserve the solutions of \eqref{eq1} for the same value of $\alpha$. 
 \end{remark}

We present  some examples of stationary hypersurfaces. Before, we introduce the following notation. A hyperbolic hyperplane in $\l^{n+1}$ with radius $r>0$ and center $p_0$ is  one of the two components of
$$
\h^n(p_0,r)=\{ p \in \l^{n+1} : \langle p-p_0, p-p_0 \rangle =-r^2\}.
$$
This hypersurface is spacelike and its mean curvature is $H=n/r$ for the orientation $N=(p-p_0)/r$.
Similarly, a pseudosphere of radius $r>0$ and centered at $p_0$ is defined by 
$$
\s_1^n(p_0,r)=\{ p \in \l^{n+1} : \langle p-p_0, p-p_0 \rangle =r^2\}.
$$
A pseudosphere is a timelike hypersurface and the mean curvature is $H=n/r$  for the orientation $N=-(p-p_0)/r$. For both hypersurfaces, if the center is the origin, we denote by  $\h^n(r)$ and $\s_1^n(r)$.

We have the following examples of stationary hypersurfaces.

\begin{enumerate}  
\item Vector  hyperplanes. These hypersurfaces are   $\alpha$-stationary hypersurfaces for all values of $\alpha$.
\item    Hyperbolic planes $\h^n(r)$. They are $\alpha$-stationary hypersurfaces with $\alpha=n$ for all values of $r$ and contained in $\mathcal{C}^+$.
\item  Pseudospheres $\s^n_1(r)$. They are $\alpha$-stationary hypersurfaces with $\alpha=n$ for all values of $r$ and contained in $\mathcal{C}^-$.
\end{enumerate}
 
In the following result, we find all hyperplanes, hyperbolic hyperplanes and pseudospheres that are $\alpha$-stationary hypersurfaces.

\begin{proposition}\label{pr1}
\begin{enumerate}
\item Vector (non-degenerate) hyperplanes are the only (non-degenerate) hyperplanes that are $\alpha$-stationary hypersurfaces. This holds for any value of $\alpha$.
\item The only hyperbolic hyperplanes that are $\alpha$-stationary hypersurfaces are:
 \begin{enumerate}
\item  Hyperbolic hyperplanes $\h^n(r)$   ($\alpha=n$) and the component of  $\h^n(p_0,r)$ with $\langle p_0,p_0\rangle=-r^2$ contained in $\mathcal{C}^-$ ($\alpha=2n$). This holds for all $r>0$.
\item The component of  $\h^n(p_0,r)$ with $\langle p_0,p_0\rangle=-r^2$ contained in $\mathcal{C}^+$ ($\alpha=-2n$). This hypersurface crosses $0$.  This holds for all $r>0$.
\end{enumerate}
\item The only pseudospheres that are $\alpha$-stationary hypersurfaces are:
\begin{enumerate}
\item  Pseudospheres $\s_1^n(r)$   ($\alpha=n$) and      $\s_1^n(p_0,r)\cap\mathcal{C}^-$ with $\langle p_0,p_0\rangle=r^2$  ($\alpha=2n$).  This holds for all $r>0$.
\item Intersections $\s_1^n(p_0,r)\cap\mathcal{C}^+$   with $\langle p_0,p_0\rangle=r^2$  ($\alpha=-2n$). The pseudosphere  crosses $0$.  This holds for all $r>0$.
\end{enumerate} 
 \end{enumerate}
\end{proposition}

\begin{proof}
The result for hyperplanes is straightforward. 

\begin{enumerate}
\item We are looking for which hyperbolic hyperplanes are $\alpha$-stationary hypersurfaces. Suppose that  $\h^n(p_0,r)\subset\mathcal{C}^+$. Then \eqref{eq1} is 
$$(2n-\alpha)\langle p,p_0\rangle+(\alpha-n)(r^2+\langle p_0,p_0\rangle)=0.$$
If $p_0=0$, this identity holds for $\alpha=n$. If $p_0\not=0$ then $\alpha=2n$ and $r^2+\langle p_0,p_0\rangle=0$. 
In such a case, the equation of the hyperbolic plane is equivalent to $\langle p,p\rangle-2\langle p,p_0\rangle=0$. One of the components crosses $0$ but this component is included in $\mathcal{C}^-$. The other component is included in $\mathcal{C}^+$. This proves the item (2a). 

Suppose now $\h^n(p_0,r)\subset\mathcal{C}^-$. A similar argument proves that  \eqref{eq1} is 
$$(2n+\alpha)\langle p,p_0\rangle-(\alpha+n)(r^2+\langle p_0,p_0\rangle)=0.$$
Notice that if $p_0=0$, then $\h^n(r)$ is included in $\mathcal{C}^+$, hence that this case is not possible. Thus $p_0\not=0$, $\alpha=-2n$ and   $r^2+\langle p_0,p_0\rangle=0$. The component of $\h^n(p_0,r)$ that crosses $0$ is contained in $\mathcal{C}^-$ but the other is contained in $\mathcal{C}^+$, which proves the item (2b). 
 
 \item Let $\s_1^n(p_0,r)$ be a pseudosphere contained in $\mathcal{C}^-$. Then \eqref{eq1} gives
  $$(n-\alpha)(r^2-\langle p_0,p_0\rangle)+(2n-\alpha)\langle p,p_0\rangle=0.$$
If $p_0=0$, then $\alpha=n$. If $p_0\not=0$, then $\alpha=2n$ and $\langle p_0,p_0\rangle=r^2$. This proves the item (3a). 

Finally, assume that $\s_1^n(p_0,r)$ is contained in $\mathcal{C}^+$. We have
 $$(n+\alpha)(r^2-\langle p_0,p_0\rangle)+(2n+\alpha)\langle p,p_0\rangle=0.$$
The case that $p_0=0$ is not possible because $\s_1^n(r)\subset \mathcal{C}^+$. Hence $\alpha=-2n$ and $\langle p_0,p_0\rangle=r^2$, proving the item (3b).

  \end{enumerate}
\end{proof}
 
A this point, it deserves to highlight two differences with the Euclidean case.  
 \begin{enumerate}
 \item Spheres of $\r^{n+1}$ centered at $0$ are stationary for $\alpha=-n$. The  analogous case in $\l^{n+1}$ is given by $\h^n(r)$ and $\s_1^n(r)$. In $\r^{n+1}$, spheres crossing $0$ are stationary for $\alpha=-2n$. This is analog to (2b) and (3b). But, in $\l^{n+1}$ there are hyperbolic hyperplanes and pseudospheres centered at points other than the origin that do not cross $0$. These examples have no counterpart in $\r^{n+1}$. 
 \item On the other hand, there are pseudospheres not centered at the origin which contain stationary parts for different values of $\alpha$, the item (3) of Prop. \ref{pr1}. This is because such pseudospheres intersect both $\mathcal{C}^-$ and $\mathcal{C}^+$.   Since they are timelike, no counterpart exists in $\r^{n+1}$.
\end{enumerate}
 
 \section{The maximum principle and inversions}\label{s22}
 
 In  this section, we give an application of the maximum principle for elliptic equations and show that inversions of $\l^{n+1}$ preserve $\alpha$-stationary hypersurfaces but change the value of $\alpha$. 
 
In case that $\Sigma$ is a spacelike hypersurface, Eq. \eqref{eq1} is elliptic and the maximum principle can be applied. The argument is the same that in the Euclidean context: see \cite[Prop. 2.5]{dl1}. The maximum principle in $\l^{n+1}$ asserts that given $\alpha\in\r$ and two spacelike hypersurfaces $\Sigma_i$, $i=1,2$, which are tangent at a common point $p\in\Sigma_1\cap\Sigma_2$, if $\Sigma_1$ lies above $\Sigma_2$ around $p$ according the the orientation by the (coincident) normal vectors at $p$, $N_1(p)=N_2(p)$, then $\widetilde{H}^1(p)\geq \widetilde{H}^2(p)$, where 
$$\widetilde{H}^i=H_i+\alpha\frac{\langle N_i,p\rangle}{|\langle p,p\rangle|},\quad i=1,2.$$
  If, in addition,  $\widetilde{H}^i$ are constant and $\widetilde{H}^1=\widetilde{H}^2$, then $\Sigma_1=\Sigma_2$ in a neighborhood of $p$. 
  
    In $\r^{n+1}$, the maximum principle allows to classify all  closed stationary hypersurfaces which are   spheres centered at $0$ ($\alpha=-n$). In the Lorentzian ambient space it is not possible to get a similar result because there are not closed spacelike hypersurfaces. However, we can use hyperbolic hyperplanes $\h^n(r)$ centered at $0$ to estimate the value $\alpha$ of some $\alpha$-stationary hypersurfaces. First, we introduce a definition. A set $S\subset\l^{n+1}$ is said to be  far away from the lightlike cone if there is $\delta>0$ such that 
  $$S\subset \mathcal{C}_\delta:=\{p\in\l^{n+1}\colon \sum_{i=1}^nx_i^2-(x_{n+1}-\delta)^2<0\}.$$
The set $\mathcal{C}_\delta$ is formed by two connected components.  Notice that the set $\sum_{i=1}^nx_i^2-(x_{n+1}-\delta)^2=0$ is a cone with vertex at $(0,\ldots,0,\delta)$. We also point out that the hyperbolic hyperplanes $\h^n(r)$ are asymptotic to $\mathcal{C}$ with two vertex $(0,\ldots,\pm r)$ for each of the components of $\h^n(r)$.      
  
  \begin{theorem} Let $\Sigma$ be a properly immersed $\alpha$-stationary spacelike hypersurface and  contained in $\mathcal{C}^+$. If $\Sigma$ is far away far from the lightlike cone, then $\alpha>n$.
  \end{theorem}
  \begin{proof}
  Since $\Sigma$ is connected, we can assume that $\Sigma$ is contained in $\mathcal{C}_\delta^+:=\mathcal{C}_\delta\cap\{x_{n+1}>0\}$.   In case that $\Sigma$ is included in  $\mathcal{C}_\delta\cap\{x_{n+1}<0\}$, the arguments are similar.   Let $\h^n_+(r)$ be the component of $\h^n(r)$ contained in $x_{n+1}>0$. Let also $r>0$ be sufficiently small so that $\h^n_+(r)\cap \mathcal{C}_\delta^+=\emptyset$. This is possible because $\h^n_+(r)$ is asymptotic to $\mathcal{C}$, $(\mathcal{C}\cap \{x_{n+1}>0\})\cap \mathcal{C}_\delta^+=\emptyset$, and the vertex of $\h^n_+(r)$ goes to the origin as $r\searrow 0$.

  Let $r\nearrow\infty$ until the first value of $r$, say $r_0$,  between $\h^n_+(r_0)$ and $\Sigma$.  The existence of $r_0$ is assured because $\h^n_+(r)$ follows being asymptotic to $\mathcal{C}$ but the vertex goes vertically to infinity. Let  $p_0\in \Sigma\cap \h^n_+(r_0)$ be a common point. Take the orientation $N(p)=p/r$ on $\h^n_+(r_0)$. Then the hypersurface $\Sigma$ lies above $\h^n_+(r_0)$. For the parameter $\alpha$,  the maximum principle at   $p_0$ yields 
  $$0=\widetilde{H}^\Sigma(p_0)\geq \widetilde{H}^{\h^n_+(r_0)}(p_0)=\frac{n}{r_0}-\alpha\frac{r_0}{r_0^2}=\frac{n-\alpha}{r_0}.$$
  This implies $\alpha\geq n$. The case $\alpha=n$ cannot occur. In such a case, $\widetilde{H}^\Sigma(p_0)=\widetilde{H}^{\h^n_+(r_0)}$ and equal to the constant $0$. The maximum principle would imply $\Sigma=\h^n_+(r_0)$, which it is not possible  because $\h^n_+(r_0)$ is not far away from the lighlike cone. This proves the result.
  \end{proof}
  
In the second part of this section we consider inversions and see that they have a well behaviour with stationary hypersurfaces. We define the inversion in $\l^{+1}$ by
 $$\phi\colon\l^{n+1}\setminus\mathcal{C}\to\l^{n+1},\quad \phi(p)=\frac{p}{\langle p,p\rangle}.$$
 We point out that $\phi^2=1_{\l^{n+1}}$ and that $\phi$ preserves the sets $\mathcal{C}^+$ (assuming $n\geq 2$) and $\mathcal{C}^-$, although it reverses the components of $\mathcal{C}^-$.
  
 Let $\Sigma$ be a hypersurface in $\l^{n+1}$, and let $\Sigma_\phi=\phi(\Sigma)$. If $v\in T_p\Sigma$, then 
 $$d\phi_p(v)=\frac{v}{\langle p,p\rangle^2}-2\frac{\langle p,v\rangle}{\langle p,p\rangle^4}p, \quad p \in \Sigma.$$
 Thus $\langle d\phi_p(v),d\phi_p(v)\rangle=\langle v,v\rangle/(\langle p,p\rangle^4)$. This proves that  $\Sigma$ is a spacelike (resp. timelike) hypersurface if and only $\Sigma_\phi$ is a spacelike (resp. timelike) hypersurface. 
 From now, all quantities of $\Sigma_\phi$ will be denoted with the subindex $\phi$. 
 
If $N$ is the unit normal vector of $\Sigma$, then 
$$N_\phi(\phi(p))=N(p)-2\frac{h(p)}{\langle p,p\rangle}p,\quad h(p)=\langle N(p),p\rangle.$$
Let $A$ be the Weingarten map of $\Sigma$. Let also $\{v_1,\ldots,v_n\}$ be an orthonormal basis of $T_p\Sigma$, where all vectors $v_i$ are spacelike except $v_n$, which is spacelike or timelike depending on whether $\Sigma$ is spacelike or timelike. If $(v_i)_\phi=d\phi_p(v_i)$, then
\begin{equation}\label{vii}
(A_\phi)_{\phi(p)}(v_i)_\phi=A_p(v_i)+2\frac{2h}{\langle p,p\rangle}v_i-\left(2\frac{\langle Av_i,p\rangle}{\langle p,p\rangle}+\frac{4h\langle p,v_i\rangle}{\langle p,p\rangle^4}\right)p.
\end{equation}
If $\varepsilon=\langle N,N\rangle=\langle N_\phi,N_\phi\rangle$, then using \eqref{vii}, one arrives at
\begin{equation}\label{hv}
\begin{split}
\frac{1}{\langle p,p\rangle^4}H_\phi&=\frac{\varepsilon\, \mbox{trace}(A_\phi)}{\langle p,p\rangle^4}= \varepsilon\sum_{i=1}^{n-1}\langle A_\phi(v_i)_\phi,(v_i)_\phi\rangle+(-1)^\varepsilon \langle A_\phi(v_n)_\phi,(v_n)_\phi\rangle\\
&=\frac{H}{\langle p,p\rangle^2}+\frac{2nh}{\langle p,p\rangle^4}.
\end{split}
\end{equation}
\begin{remark} In case that $v_i$ is a principal direction at $p$, so is $(v_i)_\phi$ at $p_\phi$ and the relation between the principal curvatures is $(\kappa_i)_\phi=\langle p,p\rangle\kappa_i+2h$.
\end{remark}

\begin{theorem} \label{ti}
Let $\Sigma$ be an $\alpha$-stationary hypersurface. Then $\Sigma_\phi$ is a $-(2n+\alpha)$-stationary hypersurface if $\Sigma\subset\mathcal{C}^+$ and a $(2n-\alpha)$-stationary hypersurface if $\Sigma\subset\mathcal{C}^-$.
\end{theorem}

\begin{proof}
It is immediate that $\langle p,p\rangle \langle p_\phi,p_\phi\rangle=1$ and $h=-\frac{h_\phi}{\langle p_\phi,p_\phi\rangle}$. From \eqref{hv}, we have 
$$H=\frac{H_\phi+2nh}{\langle p,p\rangle}=\langle p_\phi,p_\phi\rangle H_\phi-2n h_\phi.$$
Since $\Sigma$ satisfies Eq. \eqref{eq1}, then 
\begin{equation*}
\begin{split}
\langle p_\phi,p_\phi\rangle H_\phi-2n h_\phi&=H=-\alpha\frac{\langle N,p\rangle}{|\langle p,p\rangle|}
=-\alpha|\langle p_\phi,p_\phi\rangle|\langle N_\phi-\frac{2h_\phi}{\langle p_\phi,p_\phi\rangle}p_\phi,\frac{p_\phi}{\langle p_\phi,p_\phi\rangle}\rangle\\
&=\alpha|\langle p_\phi,p_\phi\rangle|\frac{h_\phi}{\langle p_\phi,p_\phi\rangle}.
\end{split}
\end{equation*}
Thus
$$H_\phi=\left(2n+\alpha\frac{|\langle p_\phi,p_\phi\rangle|}{\langle p_\phi,p_\phi\rangle}\right)\frac{\langle N_\phi,p_\phi\rangle}{\langle p_\phi,p_\phi\rangle}.$$ Then we have
$$H_\phi=\left\{\begin{array}{ll}
(2n+\alpha)\frac{\langle N_\phi,p_\phi\rangle}{|\langle p_\phi,p_\phi\rangle|}&\langle p_\phi,p_\phi\rangle>0,\\
-(2n-\alpha)\frac{\langle N_\phi,p_\phi\rangle}{|\langle p_\phi,p_\phi\rangle|}&\langle p_\phi,p_\phi\rangle<0,
\end{array}
\right.
$$
which proves the result. 
\end{proof}

We give some interesting examples of this theorem.
\begin{enumerate}
\item  A spacelike vector hyperplane $\Sigma$ is contained in $\mathcal{C}^+$ and is an $\alpha$-stationary hypersurface for any $\alpha$. Its inverse $\Sigma_\phi=\phi(\Sigma\setminus \{0\}) \cup \{0\}$ is again a stationary hyperplane. Moreover, the value of $\alpha$ changes to $-(2n+\alpha)$, which takes all values in $\r$. Consequently $\Sigma_\phi $ is a $-(2n+\alpha)$-stationary hypersurface.
\item The components of  $\h^n(r)$ reversed by $\phi$, while the value of $\alpha $ ($\alpha=n$) is preserved by Thm. \ref{ti}.
\item The pseudospheres $\mathbb{S}^n_1(r)$, as well as the value of $\alpha$, are preserved by $\phi$ according to Thm. \ref{ti}.
\item Consider a hyperplane $\Sigma$ parallel to $x_{n+1}=0$. This hyperplane is spacelike and has components both in $\mathcal{C}^+$ and in $\mathcal{C}^-$. Moreover, $\Sigma$ is maximal, so $H=0$ and $\alpha=0$ in \eqref{eq1}. Then its image by $\phi$ has three components. One of them goes to $\mathcal{C}^-\cap\{x_{n+1}<0\}$ and it is $2n$-stationary hypersurface. In fact, $\Sigma_\phi$ is the hyperbolic hyperplane of (2a) in Prop. \ref{pr1}. The components of $\Sigma$ in $\mathcal{C}^+$ go, via $\phi$, to $-2n$-stationary hypersurfaces in $\mathcal{C}^+$, which are the hyperbolic hyperplanes of (2b) in Prop. \ref{pr1}.
\item Consider a hyperplane $\Sigma$ parallel to $x_1=0$. This hyperplane is timelike and has components both in $\mathcal{C}^+$ and in $\mathcal{C}^-$. Moreover, the mean curvature $H$ of $\Sigma$ is $0$ and $\alpha=0$ in \eqref{eq1}. The  image of $\Sigma$ by $\phi$ has three components. One of them is contained in $\mathcal{C}^+$ and via $\phi$, we have a $-2n$-stationary hypersurface. This hypersurface is the pseudosphere (3b) in Prop. \ref{pr1}. The other two components of $\Sigma$ are contained in  $\mathcal{C}^+$. The inversion $\phi$ carries these components into $2n$-stationary hypersurfaces in $\mathcal{C}^-$, which are the pseudospheres of (3a) in Prop. \ref{pr1}. 
\end{enumerate}
\section{Ruled surfaces: cylindrical case}\label{s3}
In this and the next sections, we give a full classification of ruled stationary surfaces. A ruled surface $\Sigma$ in $\l^3$ is parametrized by  
\begin{equation}
X(s,t)=\gamma(s)+tw(s), \quad s\in I, t\in \r,\label{rp}
\end{equation}
where  $\gamma:I\subset \r \to \l^3$, $\gamma=\gamma(s)$, is the base curve  and   $w=w(s)\in\l^3$ is a vector field  along $\gamma$, $w(s)\not=0$, which indicates the rulings of $\Sigma$.  

 A first case is when $w=w(s)$ is constant along $\gamma$. Then, we say that $\Sigma$ is a cylindrical surface. In such a case, the corresponding classification of the stationary surfaces will be done in this section. When the surface is not cylindrical, the types of ruled surfaces depend on the causality of $w$ and $w'$ and are as follows  (see \cite{dk}):
\begin{enumerate}
\item $w$ is not lightlike:
\begin{enumerate}
\item[\text{(1a)}] $w'$ is not lightlike,
\item[\text{(1b)}] $w'$ is lightlike.
\end{enumerate}
\item $w$ is lightlike.
\end{enumerate}
Case (1) will be studied in Sect. \ref{s4}, and Case (2) in Sect. \ref{s5}.

In this section, we classify the ruled stationary surfaces of cylindrical type. First, we give some local computations which we need in the rest of sections. To be precise,  we   derive an explicit expression for \eqref{eq1} when a ruled surface $\Sigma$ is parametrized by \eqref{rp}.    

With respect to $X$,  the coefficients of the first fundamental form are
$$
E=\langle X_s,X_s\rangle, \quad F=\langle X_s,X_t\rangle,\quad G=\langle X_t,X_t\rangle.
$$
Since the following identity holds in $\l^3$
$$
\langle X_s\times X_t,X_s\times X_t\rangle=-\langle X_s,X_s\rangle\langle X_t,X_t\rangle+\langle X_s,X_t\rangle^2,
$$
the unit normal vector field is given by
$$
N=\frac{X_s\times X_t}{\sqrt{|EG-F^2|}} .
$$
The normal vector $N$ is timelike (resp. spacelike) if $\Sigma$ is spacelike (resp. timelike). Let 
$$ \langle N,N\rangle=\varepsilon\in\{-1,1\}.$$
Denote by $(\vec{u},\vec{v},\vec{w})$ the scalar triple product in $\l^3$, i.e., the determinant of the matrix formed by the vectors $\vec{u}, \vec{v}, \vec{w}\in \l^3$. The coefficients of the second fundamental form are given by 
$$
e=\frac{(X_s,X_t,X_{ss})}{\sqrt{|EG-F^2|}}, \quad f=\frac{(X_s,X_t,X_{st})}{\sqrt{|EG-F^2|}}, \quad g=\frac{(X_s,X_t,X_{tt})}{\sqrt{|EG-F^2|}}.
$$
Hence, the mean curvature $H$ is given by
$$
H=\varepsilon\frac{eG-2fF+gE}{EG-F^2}. 
$$
In the denominator of both sides of Eq. \eqref{eq1}, the factor $\sqrt{|EG-F^2|}$ cancels and Eq. \eqref{eq00} reduces to
\begin{equation}
H_1=\alpha\epsilon\varepsilon(EG-F^2)\frac{(X_s,X_t,X)}{\langle X,X\rangle}, \label{ss3}
\end{equation}
where $H_1$ is defined by 
$$
H_1:=G(X_s,X_t,X_{ss})-2F(X_s,X_t,X_{st})+E(X_s,X_t,X_{tt}).
$$

 The classification of cylindrical stationary surfaces is given in the next result.

\begin{theorem} \label{cy}
Vector planes are the only cylindrical stationary surfaces in $\l^3$.
\end{theorem}
\begin{proof}
  Let $\vec{w}$ the direction of the rulings.   We distinguish two cases according to whether $\vec{w}$ is lightlike or not.

\begin{enumerate}
\item Case $\vec{w}$ is not lightlike. We may assume   $| \vec{w} |= 1$, and that $\gamma$ is contained in a plane $\Gamma$ orthogonal to $\vec{w}$. Let   $\langle \vec{w}, \vec{w}\rangle=\delta \in \{-1,1 \}$. By reparametrizing, we can assume that $\langle \gamma',\gamma' \rangle =\mu \in \{-1,1 \}$. Hence, Eq. \eqref{ss3} is
$$
\delta (\gamma',\vec{w}, \gamma'')=-\alpha\epsilon \frac{(\gamma',\vec{w},\gamma)}{\langle \gamma, \gamma\rangle+2t\langle \gamma, \vec{w}\rangle+\delta t^2}.
$$
This equation is satisfied only if $(\gamma',\vec{w}, \gamma'')=0$ and $(\gamma',\vec{w},\gamma)=0$. The first equation implies that $\gamma$ is a straight line because $\gamma$ is contained in a plane orthogonal to $\vec{w}$.   Since $\Sigma$ is a stationary surface, then $0\in\Sigma$. 

\item Case $\vec{w}$ is   lightlike.  Then $H=0$ and Eq.  \eqref{ss3} yields $(\gamma',\vec{w},\gamma)=0$. Since $\Sigma$ is non-degenerate, we must have  $EG-F^2=-\langle\gamma',\vec{w}\rangle^2\neq 0$ on $I$. By applying a dilation followed by a rotation, we may assume  that $\vec{w}=(1,0,1)$. Consequently, $\gamma$ lies in the plane spanned by $(0,1,0)$ and $(1,0,-1)$. If $\gamma$ is a straight line 
of the form $s\mapsto (s,c,-s)$, then $\Sigma$ is a plane, and the condition $(\gamma',\vec{w},\gamma)=0$ implies $c=0$, i.e., $0 \in \Sigma$. Otherwise, we may have $\gamma(s)=(f(s),s,-f(s))$, for a smooth function $f$ on $I$. In this case, the condition $(\gamma',\vec{w},\gamma)=0$ gives $f(s)=cs$, $c\in \r$. Then $\gamma(s)=s(c,1,-c)$ is a straight line, hence $\Sigma$ is a vector plane.
\end{enumerate}
\end{proof}

\section{Ruled surfaces: the case that $w$ is not lighlike}\label{s4}

The first result assumes that $w'$ is not lightlike.

\begin{theorem}[Case $w$ and $w'$ not lightlike] \label{th-nli}
Let $\Sigma \subset \mathcal{C}^\epsilon$ be a   non-cylindrical ruled surface parametrized by \eqref{rp} where $\langle w,w\rangle=\delta$, $\langle w',w' \rangle=\mu$, and $\delta,\mu\in \{-1,1 \}$. If $\Sigma$ is an $\alpha$-stationary surface, then $\alpha = \epsilon \varepsilon$, $\langle \gamma,w\rangle =0$, and $w=w(s)$ is a geodesic of $\s_1^2(1)$. Moreover, up to a linear isometry and a dilation of $\l^3$, one of the following holds:
\begin{enumerate}
\item $\mu=1$, $w(s)=(\cos s,\sin s,0)$, $\epsilon= \varepsilon$, and either
\begin{enumerate}
\item[\textnormal{(1a)}] $\gamma(s)=(0,0,e^{\pm s})$, or
\item[\textnormal{(1b)}]  $\gamma(s)=(-\sin s,\cos s,c_1 \cosh s+c_2 \sinh s)$ with $c_1^2+c_2^2=1$.
\end{enumerate}
\item $\mu=-1$,  $\epsilon= -\varepsilon$, $\gamma(s)=(\sinh s, c_1\cos s+c_2\sin s, \cosh s)$, with $c_1^2+c_2^2=1$ and $w(s)=(\cosh s,0,\sinh s)$.
 
\end{enumerate}
\end{theorem}
\begin{proof}
Without loss of generality, we can choose $\gamma$ to be the striction curve, so that $ \langle \gamma',w'\rangle=0$. The coefficients of  the first fundamental form are
$$
E=\langle \gamma',\gamma' \rangle +\mu t^2, \quad F=\langle \gamma',w\rangle, \quad G=\delta.
$$ 
Notice that $w(s)\in \s_1^2(1)$ or $w(s)\in \h^2(1)$, according to the sign of $\delta$, and denote by $\kappa_g^w=(w'',w',w)$ the geodesic curvature of $w(s)$. We will consider the orhonormal basis $\mathcal{B}=\{w,w',w\times w' \}$, where $\langle w\times w', w\times w' \rangle =-\delta \mu$. In terms of $\mathcal{B}$, we have
\begin{align*}
\gamma'&=\delta F w-\delta \mu Q (w \times w'), \\
w''&=\delta \mu (-w+\kappa_g^w(w\times w')), \\
w\times w''&=\mu \kappa_g^w w',
\end{align*}
where $Q=(\gamma',w,w')$ is the parameter of distribution. Differentiating $\gamma'$ with respect to $s$, we have
$$
\gamma''=\delta F'w +\delta (F-Q\kappa_g^w)w'-\delta \mu Q' (w\times w').
$$
Set $\beta=-\varepsilon \epsilon \in \{-1,1\}.$ Hence, Eq. \eqref{ss3} becomes
$$
\frac{-Q(F+Q\kappa_g^w)-\delta Q' t +\delta \kappa_g^wt^2}{\mu (\delta t^2-Q^2)}=-\alpha\beta\frac{(\gamma'+tw',w,\gamma)}{\langle \gamma,\gamma\rangle+2t\langle \gamma,w\rangle+\delta t^2}.
$$ 
This equation can be written  of the form $$\sum_{n=0}^4A_n(s)t^n=0.$$ 
Thus all coefficients   $A_n(s)$ vanish identically. For $n=3,4$, a computation yields 
$$
A_3=Q'-\alpha\beta\delta\mu(w',w,\gamma), \quad A_4=\kappa_g^w.
$$
From $A_4=0$,  $w=w(s)$ is a geodesic of $\s_1^2(1)$ or of $ \h^2(1)$. Also, it follows that $w\times w'$ is constant along $\gamma$. Since the condition $Q=0$ yields that $\Sigma$ is a vector plane, we will assume that $Q\neq 0$.

We write $\gamma$ in terms of $\mathcal{B}$, obtaining 
$$
\gamma=cw+aw'+b(w\times w'),
$$
where $a$, $b$, and $c$ are smooth functions defined on $I$.  Differentiating and using $\langle \gamma',w'\rangle=0$, we find $c=-a'$ on $I$. Consequently, 
\begin{align*}
\gamma&=-a'w+aw'+b(w\times w'), \\
\gamma'&=-(a''+\delta\mu a)w+b'(w\times w'),
\end{align*}
where $-(a''+\delta\mu a)=\delta F$ and $Q=-\delta\mu b'$. Thus, from $A_3=0$, we conclude
\begin{equation}
 b''+\alpha\beta\delta\mu b =0. \label{ode1}
\end{equation} 
On the other hand, the remaining coefficients $A_n$ are  
\begin{align*}
A_0&=QF\langle \gamma,\gamma \rangle +\alpha \beta \mu Q^2(\gamma',w,\gamma), \\
A_1&=2FQ\langle \gamma,w \rangle +\delta Q' \langle \gamma,\gamma \rangle+\alpha \beta \mu Q^2 (w',w,\gamma), \\
A_2&=2\delta Q'\langle \gamma,w \rangle +\delta QF-\alpha \beta \delta \mu  (\gamma',w,\gamma).
\end{align*}
We separate two cases:

\begin{enumerate}
\item  Case $F=\langle \gamma',w \rangle =0$. From $A_0=0$, we have $(\gamma',w,\gamma)=0$. By considering in $A_2=0$, we obtain that either $Q'=0$ or $\langle \gamma,w \rangle =0$. If $Q'=0$, then from \eqref{ode1} we arrive at $b=0$, i.e., $\Sigma$ is a plane, which is not our case. Thus,   $\langle \gamma,w \rangle =0$, which yields   $\gamma=b (w\times w')$. By using $A_1=0$, we find $bb''+\alpha \beta\delta\mu b'^2=0$. Comparing Eq. \eqref{ode1}, we obtain $b^2-\delta\mu b'^2=0$. Here, it must hold $\delta\mu=1$; otherwise we would have $b=0$ or $Q=0$, a contradiction. Since $\langle w,w'\rangle =0$, we arrive at $\delta=\mu=1.$ In addition, due to $b'=\pm b$, using \eqref{ode1}, we conclude $\alpha=-\beta$. This implies that either $b(s)=c_0e^s$ or $b(s)=c_0e^{-s}$, where $c_0 \in \r$ and $c_0 \neq 0$. Up to a dilation, we may assume that $c_0=1$. On the other hand, $EG-F^2=t^2-Q^2$ and $\langle X,X\rangle=\langle \gamma,\gamma\rangle+t^2=-b^2+t^2$, which implies $\epsilon=\varepsilon$. Therefore, after solving $w''+w=0$ and applying a linear isometry of $\l^3$, we find $w(s)=(\cos s, \sin s, 0)$, which completes the proof of the item (1a) of the theorem.

\item Case $F=\langle \gamma',w \rangle \neq 0$. If $Q'=0$, then from $A_3=0$ it follows that $(w',w,\gamma)=0$, that is, $b=0.$ This contradicts with $Q\neq 0$. From $A_0=0$ and $A_2=0$, we have $\langle \gamma, \gamma \rangle=-\alpha\beta\mu\frac{Q}{F}(\gamma',w,\gamma)$ and 
\begin{equation*}
\alpha\beta\mu(\gamma',w,\gamma)=2Q'\langle \gamma, w \rangle+QF,
\end{equation*}
respectively. This yields   
\begin{equation}
\langle \gamma, \gamma \rangle=-\frac{Q}{F}(2Q'\langle \gamma, w \rangle+QF). \label{qf}
\end{equation}
Inserting this in $A_1=0$, we obtain 
$$\langle \gamma, w \rangle (F^2-\delta Q'^2)=0.$$ 
We investigate the resulting cases separately.

\begin{enumerate}
\item Case $\langle \gamma, w \rangle = 0$. It follows that $a(s)=a_0\not=0$ is constant, and that $F = -\mu a_0$. Up to a dilation, we may assume that $a_0=1$. From the equations $A_0 = 0$ and $A_1=0$, we obtain $1= \delta b^2 + \alpha \beta \mu b'^2$  and $1= \delta b^2 - \mu b'^2$, respectively. Both equations imply  $\alpha \beta = -1$. Hence, there are two possibilities: $(\delta,\mu)=(1,1)$ and $(\delta,\mu)=(1,-1)$. If $\mu=1$, then, similarly to Case 1, we arrive at item (1b) of the theorem. Otherwise, $\mu=-1$,  solving $w''+\mu w=0$ and applying a linear isometry of $\l^3$ we obtain $w=(\cosh s, 0 , \sinh s).$ Consequently, after solving Eq.  \eqref{ode1}, we establish item (2) of the theorem. Similar to Case 1, it is direct to see that $EG-F^2=-\langle X,X\rangle$, i.e. $\epsilon=-\varepsilon$.

\item Case  $\langle \gamma, w \rangle \neq 0$. It follows $F^2-\delta Q'^2=0$, which requires $\delta=1$, that is, $F=\pm Q'$. Without loss of generality, suppose $F=Q'$, that is
\begin{equation}
a''+\mu a=\mu b''. \label{amub}
\end{equation}
On the other hand, from $A_2=0$, we conclude
\begin{equation}
2a'b+bb'-ab'=0. \label{ab}
\end{equation}
We will show that this case leads to a contradiction. The argument will be divided into cases according to the values of $\mu$ and $\alpha\beta$. Although the work is a bit tedious, the reasoning in each case is identical, where first we   solve the ODEs \eqref{ode1} and \eqref{amub}, and then substitute the solutions into \eqref{ab} obtaining an equation of functions on $s$ which are linearly independent. In particular, all coefficients of these functions must vanish identically. An analysis of theses coefficients gives the desired contradiction. Set $k=\alpha\beta$. To simplify the notation, let $c_i$ ($1 \leq i \leq 4$) denote the constants of integration that will appear below.
\begin{enumerate}
\item Case $\mu=1$ and $k=1$. The solutions of \eqref{ode1} and \eqref{amub} are
\begin{equation*}
\begin{split}
b(s)&=c_1\cos s+c_2 \sin s,\\
a(s)&=\frac12 ((-c_1+2c_3+c_2s)\cos s+(-c_1s+2c_4)\sin s).
\end{split}
\end{equation*}
Substituting   in \eqref{ab} yields a trigonometric equation whose coefficients must vanish. The coefficient of the term $s \cos^2 s$ is $-c_1^2-\frac12 c^2$, which is nonzero, leading to a contradiction.

\item Case $\mu=1$, $k >0$, and $k\neq1$. Now we have
\begin{equation*}
\begin{split}
 b(s)&=c_1\cos (\sqrt{k}s)+c_2 \sin (\sqrt{k}s),\\
a(s)&=\frac{k }{k-1}(c_1 \cos (\sqrt{k} s)+c_2 \sin (\sqrt{k} s))+c_3 \cos s+c_4 \sin s.
\end{split}
\end{equation*}
By substituting   in \eqref{ab}, we arrive at
$$
\begin{aligned}
0=& \sqrt{k}(-c_1^2+c_2^2)\sin(2\sqrt{k}s)
+ 2\sqrt{k}c_1c_2\cos(2\sqrt{k}s) \\
&+ (\sqrt{k}-1)(c_1c_3-c_2c_4)\sin((\sqrt{k}+1)s)
+ (\sqrt{k}+1)(c_1c_3+c_2c_4)\sin((\sqrt{k}-1)s) \\
&+ (\sqrt{k}+1)(c_1c_4-c_2c_3)\cos((\sqrt{k}-1)s)
- (\sqrt{k}-1)(c_1c_4+c_2c_3)\cos((\sqrt{k}+1)s).
\end{aligned}
$$
From the coefficients of the terms of $\sin(2\sqrt{k}s)$ and $\cos(2\sqrt{k}s)$, we find $c_1=c_2=0$, which is not possible.

\item Case $\mu=1$, $k <0$. We have 
\begin{equation*}
\begin{split}
b(s)&=c_1e^{\sqrt{-k}s}+c_2 e^{-\sqrt{-k}s},\\ 
a(s)&=\frac{k }{k-1}(c_1 e^{\sqrt{-k} s}+c_2 e^{-\sqrt{-k} s})+c_3 \cos (s)+c_4 \sin (s).
\end{split}
\end{equation*}
Using these expressions in \eqref{ab}, we obtain 
$$
\begin{aligned}
0=&
\sqrt{-k}\frac{2k-1}{k-1}( c_1^{2} e^{2\sqrt{-k}s}-c_2^{2} e^{-2\sqrt{-k}s}) \\
&+(c_1c_3(-2\sin s - \sqrt{-k}\cos s)+c_1c_4(2\cos s - \sqrt{-k}\sin s)) e^{\sqrt{-k}s} \\
&+(c_2c_3(-2\sin s +\sqrt{-k}\cos s)+c_2c_4(2\cos s + \sqrt{-k}\sin s))e^{-\sqrt{-k}s}.
\end{aligned}
$$
This leads to a contradiction, because the coefficients of the terms $e^{2\sqrt{-k}s}$ and $e^{-2\sqrt{-k}s}$ cannot both vanish simultaneously.

\item Case $\mu=-1$ and $k=1$. We have 
\begin{equation*}
\begin{split}
b(s)&=c_1e^s+c_2 e^{-s},\\
a(s) &= c_3 e^s + c_4  e^{-s} - \frac{c_1}{2} s e^s + \frac{c_2}{2} s e^{-s}.
\end{split}
\end{equation*}
Inserting these expressions in \eqref{ab}, we obtain
$$
(c_1 c_3 - \frac{c_1^2}{2}  s ) e^{2s}
+( 3c_2 c_3 - 3c_1 c_4 - 3c_1 c_2 s )
+ ( -c_2 c_4 - \frac{c_2^2}{2}  s ) e^{-2s}=0,
$$
which gives the contradiction $c_1=c_2=0$.

\item Case $\mu=-1$ and $k>0$, $k\neq 1$. We have
\begin{equation*}
\begin{split}
b(s)&=c_1e^{\sqrt{k}s}+c_2 e^{-{\sqrt{k}s}},\\
a(s)&= c_3e^{s} + c_4e^{-s}- \frac{k}{k-1}( c_1e^{\sqrt{k}s} + c_2e^{-\sqrt{k}s} ).
\end{split}
\end{equation*}
Considering these in \eqref{ab}, we obtain
$$
\begin{aligned}
0 &= (2 - \sqrt{k})\,c_1 c_3\, e^{(1+\sqrt{k})s}
+ (2 + \sqrt{k})\,c_2 c_3\, e^{(1-\sqrt{k})s} \\
&\quad - (2 + \sqrt{k})\,c_1 c_4\, e^{(-1+\sqrt{k})s}
- (2 - \sqrt{k})\,c_2 c_4\, e^{-(1+\sqrt{k})s} \\
&\quad - \frac{\sqrt{k}}{k-1} c_1^2 e^{2\sqrt{k} s}
+ \frac{\sqrt{k}}{k-1} c_2^2 e^{-2\sqrt{k} s},
\end{aligned}
$$
where the coefficients must vanish. But this gives  $c_i = 0$, a contradiction.

\item Case $\mu=-1$ and $k<0$. We have 
\begin{equation*}
\begin{split}
b(s)&=c_1\cos (\sqrt{-k}s)+c_2 \sin(\sqrt{-k}s),\\
a(s) &= c_3 e^{s} + c_4 e^{-s}+ \frac{k}{k-1} ( c_1 \cos( \sqrt{-k} s )+c_2 \sin( \sqrt{-k}s ) .
\end{split}
\end{equation*}
Substituting these into \eqref{ab} and focusing only on the purely trigonometric terms, we find  
$$
\frac{\sqrt{-k}(1-2k)}{k-1} ( \frac12(c_{1}^2 - c_{2}^2) \sin(2 \sqrt{-k}s) - c_{1} c_{2} \cos(2 \sqrt{-k} s)).
$$
This gives $c_1=c_2=0$,  a contradiction.

\end{enumerate}
\end{enumerate}
\end{enumerate}
\end{proof}



\begin{theorem}[Case $w$ not lightlike and $w'$ lightlike] \label{th-nli2}
Vector planes are the only     ruled surfaces parametrized by \eqref{rp} where   $w$   is not lightlike and $w'$ is lightlike. 
\end{theorem}

\begin{proof}
By contradiction, we assume that the ruled surface parametrized by \eqref{rp} is non-planar. We may suppose that $w$ is orthogonal to $\gamma$ such that $\langle w, w \rangle = 1$ and $\langle w',w' \rangle=0$. Hence, the curve $w=w(s)$ is contained in $\s_1^2(1)$. Differentiating $\langle w,w' \rangle=0$ and $\langle w',w' \rangle=0$, we obtain $\langle w,w'' \rangle=\langle w',w'' \rangle=0$. We then conclude that $w''$ is tangent to $ \s_1^2(1)$. As $ \s_1^2(1)$ is $2$-dimenisonal and $w''$ is orthogonal to $w'$, then $w''$ is spanned by $w'$. After a reparametrization, we may assume that $w''=0$, i.e., $w$ describes a straight line contained in $\s_1^2(1)$. Set $w(s)=p_0+s\vec{v}$, where $\langle p_0,p_0\rangle=1$ and $\langle p_0,\vec{v}\rangle=\langle \vec{v},\vec{v}\rangle=0$. After using a suitable rotation, we may take $ p_0=(1,0,0)$, which yields that $\vec{v}=(0,v_2,v_3)$. Since $v_2^2-v_3^2=0$, we conclude $v_2=\pm v_3$, that is, $\vec{v}=(0,\pm v_3,v_3)$. By applying a dilation and, if necessary, a reflection about $xy$-plane or $xz-$plane, we may assume $\vec{v}=(0,1,1)$. Finally, by reversing the parameter $s\mapsto -s$, we arrive $w(s)=(1,-s,-s)$. It is direct to conclude that $w\times w'=-w'$. 

Now, we set $\gamma(s) = (x(s), y(s), z(s))$. Under the assumptions above, it follows that
\begin{equation}
x'-s(y'-z')=0. \label{xyz}
\end{equation}
On the other hand, a direct calculation shows that Eq.  \eqref{ss3} is 
$$
(\gamma',w,\gamma'')+t(w',w,\gamma'')=-\alpha\beta(\langle \gamma',\gamma'\rangle+2t\langle \gamma',w'\rangle)\frac{(\gamma',w,\gamma)+t(w',w,\gamma)}{\langle \gamma,\gamma\rangle+2t\langle \gamma,w\rangle+t^2},
$$
where $\beta=-\epsilon\varepsilon$. This equation can be written as $\sum_{n=0}^3 A_n(s) t^n = 0$. Thus  the functions $A_n(s)$ vanish identically. By computing $A_2$ and $A_3$, we obtain
$$
\begin{aligned}
A_2&=(\gamma',w,\gamma'')+2\langle\gamma, w \rangle(w',w,\gamma'')+2\alpha \beta \langle\gamma', w' \rangle (w',w,\gamma), \\
A_3&=(w',w,\gamma'').
\end{aligned}
$$
From $A_3=0$, we get $\langle w',\gamma'' \rangle =0$. Using this last expression and $w'' = 0$, we find that $\langle \gamma', w' \rangle = -m$, where $m$ is a constant. We claim that $m \neq 0$: otherwise, if  $m=0$, then we find $y'=z'$. Substituting in \eqref{xyz} gives $x'=0$, that is, $\gamma'$ becomes a lightlike vector field. This yields $EG-F^2=0$, which it is a contradiction. Therefore we obtain $y'-z'=m$ and
\begin{equation}
 y(s)-z(s)=ms+a, \quad a\in \r. \label{xyz1}
\end{equation}
By Eqs. \eqref{xyz} and \eqref{xyz1}, we have $x(s)=\frac{m}{2}s^2+b$, $b\in \r$.
Also, from $A_2=0$ it follows
\begin{equation}
(\gamma',w,\gamma'')+2\alpha \beta m (w',w,\gamma)=0. \label{xyz2}
\end{equation}
Considering   \eqref{xyz} and \eqref{xyz1}, we obtain 
$$
(\gamma',w,\gamma'')=m(z''+ms), \quad (w',w,\gamma)=-(ms+a).
$$
Substituting in \eqref{xyz2}, we have $z''(s)=m(-1+2\alpha\beta)s+2a\alpha\beta$. By integrating, we conclude
\begin{equation*}
\begin{split}
y(s)&=\frac{m}{6}(-1+2\alpha\beta)s^3+a\alpha \beta s^2+cs+d, \\
 z(s)&=\frac{m}{6}(-1+2\alpha\beta)s^3+a\alpha\beta s^2+(c-m)s+d-a.
\end{split}
\end{equation*}
The coefficient $A_0$ is now
$$
A_0=\langle \gamma, \gamma \rangle (\gamma',w,\gamma'')+\alpha\beta \langle \gamma', \gamma' \rangle (\gamma',w,\gamma)=0.
$$
By substituting the expressions for $x$, $y$, and $z$, we obtain a polynomial equation on $s$ of degree $5$ with leading coefficient 
\begin{equation}
\frac{m^{4}\alpha\beta}{6}(8(\alpha\beta)^2+10\alpha\beta-1). \label{lead}
\end{equation}
Moreover, for the coefficient $A_1$, we have the following
$$
A_1=2\langle \gamma, w \rangle (\gamma',w,\gamma'')-2\alpha \beta m (\gamma',w,\gamma)+\alpha\beta\langle \gamma',\gamma' \rangle (w',w,\gamma)=0.
$$
Similarly, by substituting the expressions for $x$, $y$, and $z$, we arrive at a polynomial equation on $s$ of degree $3$ whose leading coefficient is $-\frac{m^{3}\alpha\beta}{3}(7+10\alpha\beta)$. Since this coefficient must be zero, we find $\alpha\beta=-\frac{7}{10}$. However, inserting this value of $\alpha\beta$ in \eqref{lead} leads to a contradiction.
\end{proof}

\section{Ruled surfaces: the case that $w$ is lightlike}\label{s5}
In our last result, we classify  ruled stationary surfaces with lightlike rulings.

\begin{theorem}[Case $w$   lightlike] \label{th-li}
Let $\Sigma\subset\mathcal{C}^\epsilon$ be a ruled surface parametrized by \eqref{rp}, where $w(s)$ is lightlike. If $\Sigma$ is an $\alpha$-stationary surface, then $H \neq 0$, $(w,w',w'') \neq 0$, and one of the following occurs:
\begin{enumerate}
 \item $\alpha=2$ and, up to a dilation, $\gamma(s)=w'(s)$. Hence, $\Sigma\subset \s^2_1(1)$.
 \item $\alpha=-4\epsilon$, $\langle \gamma ,w\rangle \neq 0$, and either
\begin{enumerate}
\item[\text{(2a)}] the base curve $\gamma$ is a straight line contained in $\mathcal{C}$   passing through the origin, or 
\item [\text{(2b)}] we have the identity
\begin{equation}
\langle \gamma,\gamma \rangle (w',w,\gamma)-2\langle \gamma,w\rangle (\gamma',w,\gamma)=0, \label{th-li2}
\end{equation}
where each term is nonzero.
\end{enumerate}
\end{enumerate}
\end{theorem}

\begin{proof}
Since  $w(s)$ is a lightlike and $\langle w,w'\rangle=0$, then   $w'(s)$ is spacelike. We have $EG-F^2=-\langle \gamma',w \rangle^2$, hence $\langle\gamma',w\rangle\not=0$.   By reparametrizing the surface, we may assume that $\langle w',w' \rangle =1$ and $\langle \gamma',w' \rangle =0$. Therefore, Eq. \eqref{ss3} is 
\begin{equation*}
2(\gamma',w,w')=\alpha\epsilon \langle \gamma',w \rangle\frac{ (\gamma'+tw',w,\gamma)}{\langle \gamma,\gamma\rangle+2t\langle \gamma,w\rangle},
\end{equation*}
which can be written as $A_0(s)+A_1(s)t=0$. The computation of $A_0$ and $A_1$ gives
\begin{eqnarray*}
A_0&=&2\langle \gamma,\gamma\rangle(\gamma',w,w') -\alpha\epsilon \langle \gamma',w \rangle (\gamma',w,\gamma),\\
A_1&=&4\langle \gamma,w\rangle(\gamma',w,w')-\alpha \epsilon\langle \gamma',w \rangle(w',w,\gamma).
\end{eqnarray*}
From $A_1=0$, it follows $(4\langle \gamma,w\rangle\gamma'+\alpha\epsilon \langle \gamma',w \rangle\gamma,w,w')=0$, which implies the existence of two smooth functions $a,b$ on $I$ such that
\begin{equation}
4\langle \gamma,w\rangle\gamma'+\alpha \epsilon\langle \gamma',w \rangle\gamma=aw+bw'. \label{l-1}
\end{equation}
Multiplying by $w$, we obtain $(4+\alpha\epsilon)\langle \gamma,w\rangle\langle \gamma',w \rangle=0$. 
We separate two cases:

\begin{enumerate}
\item Case $\alpha\epsilon \neq-4$. Then, $\langle \gamma,w\rangle=0$, and from $A_1=0$ we obtain $(w',w,\gamma)=0$. If $\gamma$ is linearly dependent on $w$, then the contradiction $\Sigma \subset \mathcal{C}$ is obtained. Hence, we have $\gamma=cw'$, for a smooth function $c$ on $I$. Since $\langle \gamma',w' \rangle =0$, the function $c$ is a nonzero constant $c_0$ such that $\gamma'=c_0w''$. Writing in $A_0=0$, we arrive at $c_0^3(2+\alpha\epsilon)(w'',w,w')=0$. 
We will see that $(w'',w,w')=0$ is not possible. In such a case we have $H=0$. Since $\gamma'=c_0w''$ and $\gamma$ is a regular curve, we have $w''\neq 0$. Then, we have $w''=dw+ew'$, for smooth functions $d$ and $e$ on $I$.  However, this leads to a contradiction: differentiating $\langle w',w' \rangle=1$ and $\langle w',w \rangle=0$, we find $\langle w'',w' \rangle = 0$ and $\langle w'',w  \rangle = -1$, respectively. These are incompatible with the expression $w''=dw+ew'$. This implies that  $\alpha=-2\epsilon$.  Moreover, since $\langle X,X\rangle=c_0^2$, we obtain $\epsilon =-1$, which proves the first item.

\item Case  $\alpha\epsilon = -4$. Due to the previous item, we have $\langle \gamma,w\rangle\neq 0$. We distinguish two subcases.
\begin{enumerate}
\item Case  $\langle \gamma, \gamma \rangle = 0$. It follows  $\langle \gamma, \gamma' \rangle =0$. By $A_0=0$, we have $(\gamma',w,\gamma)=0$, that is, there are functions $a,b,c$ such that $a\gamma+b\gamma'+cw=0$. Multiplying by $\gamma$, we conclude $c=0$. Since $\gamma$ is a regular curve, we obtain $\gamma$ is straight line contained in $\mathcal{C}$ and passing through the origin.

\item Case  $\langle \gamma, \gamma \rangle \neq 0$. In this case, we claim that $(\gamma',w,w') \neq 0$, or equivalently, $H\neq 0$. Otherwise, since $\langle \gamma',w'\rangle \neq 0 $, we would obtain $\langle \gamma',w\rangle =0 $. Therefore, the combination $2\langle \gamma,w\rangle A_0-\langle \gamma,\gamma \rangle A_1=0$ gives \eqref{th-li2}, completing the proof.
\end{enumerate}
\end{enumerate}
\end{proof}

We give two explicit examples of the items (1) and (2b) of Thm. \ref{th-li}.

\begin{example} Let $w(s)=(\sinh(s),1,\cosh(s))$ and consider the surface  
$$X(s,t)=w'(s)+tw(s)=( \cosh (s)+t \sinh (s),t,\sinh (s)+t\cosh(s))
).$$
Then $\langle X,X\rangle=1$ and $\Sigma\subset\s^2_1\subset\mathcal{C}^-$. Consequently, $\Sigma$ is a  $2$-stationary surface. 
\end{example}

 \begin{example}
Consider $0<s_0<\frac{\pi}{2}$ and
\begin{equation*}
\begin{split}
\gamma(s)&=( \sin (s), -\cos (s), \tan(\frac{s}{2})),\\
w(s)&= (\cos (s), \sin (s),1),
\end{split}
\end{equation*}
where $s\in [s_0,\frac{\pi}{2})$. Let $\Sigma$ be a ruled surface parametrized by \eqref{rp}. First, we determine an interval for the parameter $t$ such that $\Sigma \subset \mathcal{C}^+$ or $\Sigma \subset \mathcal{C}^-$. Then, we have
$$
\langle X(s,t),X(s,t)\rangle=1-2t\tan(\frac{s}{2})-\tan^2(\frac{s}{2}).
$$
Since the function $\tan (\frac{s}{2})$ is decreasing on $[s_0,\frac{\pi}{2})$, we deduce $\tan(\frac{s}{2}) \in [\tan(\frac{s_0}{2}), 1)$. Setting $u=\tan(\frac{s}{2})$, $u_0=\tan(\frac{s_0}{2})$ and $F(u)=\frac{1-u^2}{2u}$, we get
$$
F(u)\in (0,\frac{1-u_0^2}{2u_0}], \quad u\in [u_0,1).
$$
If $\epsilon=-1$, i.e. $\langle X,X\rangle >0$, then $t< F(u) $, for every $u\in [u_0,1)$. Since the lower bound of $F(u)$ is zero, it follows that $\Sigma \subset \mathcal{C}^-$ for all $t \leq 0$. Otherwise, $\langle X,X\rangle <0$, then $t>F(u) $, for every $u\in [u_0,1)$. Hence, $\Sigma \subset \mathcal{C}^+$ for all $t>\frac{1-u_0^2}{2u_0}$. Consequently, $\Sigma \cap \mathcal{C}=\emptyset$ as long as $t\in (-\infty,0]\cup (\frac{1-u_0^2}{2u_0},\infty)$.

On the other hand, a direct calculation yields
\begin{equation*}
\begin{split}
\langle \gamma(s),\gamma(s)\rangle &= 1-\tan^2(\frac{s}{2}), \quad \langle \gamma(s),w(s)\rangle =-\tan(\frac{s}{2}),\\ 
(w'(s),w(s),\gamma(s))&= -\tan(\frac{s}{2}), \quad (\gamma'(s),w(s),\gamma(s))= \frac12 (1-\tan^2(\frac{s}{2})). \\
\end{split}
\end{equation*}
The above terms are all nonzero on $ [s_0,\frac{\pi}{2})$. Therefore, Eq. \ref{th-li2} is satisfied. Finally, $X(s,t)=\gamma(s)+tw(s)$ describes a non-cylindrical, non-planar ruled $\alpha$-stationary surface with lightlike rulings, where $\alpha$ is $4$ or $-4$ if $t\in (-\infty,0]$ or $t\in(\frac{1-u_0^2}{2u_0},\infty)$, respectively.

\end{example}
\section*{Acknowledgement}

 Rafael L\'opez   has been partially supported by MINECO/ MICINN/FEDER grant no. PID2023-150727NB-I00, and by the ``Mar\'{\i}a de Maeztu'' Excellence Unit IMAG, reference CEX2020-001105- M, funded by MCINN/AEI/ 10.13039/501100011033/ CEX2020-001105-M.


\end{document}